\documentclass[12pt]{article}
\usepackage[centertags]{amsmath}
\usepackage{amsfonts}
\usepackage{amssymb}
\usepackage{amsthm}
\input{amssym.def}

\newtheorem{Definition}{Definition}[section]
\newtheorem{Theorem}[Definition]{Theorem}
\newtheorem{Lemma}[Definition]{Lemma}

\newtheorem{Corollary}[Definition]{Corollary}

\newtheorem{Example}[Definition]{Example}

\newcommand{\lc}{\mathcal{L}}
\newcommand{\rc}{\mathcal{R}}
\newcommand{\hc}{\mathcal{H}}
\newcommand{\jc}{\mathcal{J}}

\title{\Large \bf Minimal bi-ideals in regular and completely regular ordered semigroups}
\author{ K. Hansda \\
\footnotesize{Department of Mathematics, Visva-Bharati
University,}\\
\footnotesize{Santiniketan-731235, West Bengal, India}\\
 \footnotesize{kalyan.hansda@visva-bharati.ac.in}}

\begin{document}
\date{}
\maketitle
\begin{abstract}{\footnotesize}
Here we  characterize regular and completely regular  ordered
semigroups by their minimal bi-ideals. A minimal bi-ideal is
expressed as a product of a minimal right ideal and a minimal left
ideal. Furthermore, we show that every bi-ideal in a completely
regular ordered semigroup is minimal and hence  a regular ordered
semigroup $S$ is completely regular if and only if $S$ is union its
of minimal bi-ideals.
\end{abstract}

{\it Key Words and phrases:} ordered semigroup, regular, completely
regular, bi-ideal, t-simple,  ordered idempotents, bi-simple.

{\it 2010 Mathematics subject Classification:} 20M10; 06F05.

\section{Introduction}
The notion of regularity in ordered semigroups is derived by  N.
Kehayopulu \cite{ke921}. As well as ring theory regularity plays a
noticeable role in ordered semigroups. T. Saito
\cite{Saito2}\cite{Saito3} studied systematically ordered regular,
completely regular ordered semigroups. Success attained by this
school characterizing regularity on ordered semigroups are either in
the semilattice and complete semilattice decompositions into
different types of simple components, viz. left, $t-$, $\sigma$,
$\lambda$-simple etc. or in its ideal theory.

Here our aim is to study regular and completely regular ordered
semigroups by minimality of their bi-ideals. N. Kehayopulu
\cite{ke92} introduced the notion of bi-ideal  in an ordered
semigroup. Mathematicians like Lee, Kang \cite{lt} and others
studied these type of ideals in various ways. Author \cite{bh2}
characterized  bi-ideals in Clifford and left Clifford ordered
semigroup. Cao and XU \cite{Cao00} described minimal and maximal
left ideals in ordered semigroup.  Xu and Ma \cite{XM} studied
minimality of bi-ideals in an ordered semigroup  and characterized
t-simplicity of ordered semigroups by  minimality of their
bi-ideals.

In this paper we use this technique of  minimality of bi-ideals  to
study the structure  of  completely regular ordered semigroups. Our
paper organized as follows. Basic definitions and some known reults
of ordered semigroups have been given  in Section 2. Section 3 is
devoted to characterizing the regular and completely regular ordered
semigroups by  minimality of their bi-ideals.

\section{Preliminaries:}
In this paper $\mathbb{N}$ will provide the set of all natural
numbers. An ordered semigroup is a partially ordered set $S$, and at
the same time a semigroup $(S, \cdot)$ such that $( \forall a , \;b
, \;x \in S ) \;a \leq b$ implies  $xa\leq xb \;\textrm{and} \;a x
\leq b x$. It is denoted by $(S,\cdot, \leq)$.

For an ordered semigroup $S$ and $H \subseteq S$, denote $$(H] :=
\{t \in S : t \leq  h, \;\textrm{for some} \;h\in  H\}.$$

Let $I$ be a nonempty subset of an ordered semigroup $S$. $I$ is a
left (right) ideal of $S$, if $SI \subseteq I \;( I S \subseteq I)$
and $(I]= I$. $I$ is an ideal of $S$ if it is both  left and  right
ideal of $S$. $S$ is left (right) simple if it has no non-trivial
proper left (right) ideal. Similarly we define simple ordered
semigroups. $S$ is called t-simple ordered semigroup if it is both
left and right simple. Due to Kehayopulu   an ordered semigroup $S$
is called an  regular \cite{ke921} (completely regular \cite{ke92} )
ordered semigroup if for every $a \in S, \;a \in (aSa] \;(a \in
(a^2Sa^2])$.

A subsemigroup  $B$ of $S$ is called a  bi-ideal \cite{ke92} if $
BSB \subseteq B$ and $(B]= B$.  The principal  left ideal, right
ideal \cite{Ke2006}, ideal and   bi-ideal \cite{ke92} generated by
$a \in S$ are denoted by $L(a), \;R(a), \;I(a), \;B(a)$
respectively. It is easy to check that
$$L(a)= (a \cup Sa], \;R(a)= (a\cup aS], \;I(a)= (a\cup Sa \cup aS \cup SaS] \;  and  \;B(a)=(a \cup a^2 \cup aSa],$$
and  if moreover  $a$ is  regular then $L(a)= (  Sa], \;R(a)= ( aS],
\;I(a)= ( Sa \cup aS \cup SaS]$ and $B(a)=( aSa]$.  Kehayopulu
\cite{Ke2006} defined Greens relations $\lc, \;\rc, \;\jc
\;\textrm{and} \;\hc$ on an ordered semigroup $S$ as follows:
$$ a \lc
b   \; if   \;L(a)= L(b), \; a \rc b   \; if   \;R(a)= R(b), \; a
\jc b \; if   \;I(a)= I(b), \;\textrm{and} \;\hc= \;\lc \cap
\;\rc.$$

These four relations  are equivalence relations on $S$.

\begin{Theorem}\cite{Cao02}\label{1}
An ordered semigroup $S$ is regular if and only if for every right
ideal $R$ and left ideal $L$ of $S, \;(RL]= R \cap L$.
\end{Theorem}

\begin{Theorem}\cite{ke02}\label{.3}
Let $S$ be regular ordered semigroup, and $B$ a bi-ideal of $S$.
Then $B= (BSB]$.
 \end{Theorem}
 By an ordered idempotent \cite{bh1}  in an ordered semigroup $S$, we shall mean an
element $e\leq S$ such that $e \leq e^2$. The set of all ordered
idempotents in $S$ will denoted by $E_\leq(S)$.

For example consider the ordered semigroup $(\mathbb{R}^+, \cdot,
\leq )$. Then $(\mathbb{R}^+, \cdot)$ is not regular as a semigroup
but it is ordered regular, as for example $2 \leq 2 \cdot 2 \cdot
2$. Again $1$ is the only idempotent in the semigroup
$(\mathbb{R}^+, \cdot)$ where as each natural number $n$ is an
ordered idempotent.

In an ordered semigroup $S$, every left (right) ideal a quasi-ideal
and every quasi-ideal is a bi-ideal. Keeping in mind that every
t-simple ordered semigroup is a t-simple ordered semigroup, we
restate the result of Kehayopulu \cite{ke02}.
\begin{Theorem}\cite{ke02}\label{0}
An ordered semigroup $S$ is t-simple ordered semigroup if and only
if it has no proper bi-ideal.
\end{Theorem}

\begin{Theorem}\cite{bh1}\label{2}
An ordered semigroup $S$ is completely regular if and only if $S$ is
union of t-simple ordered semigroups.
\end{Theorem}
\begin{Lemma}\label{cr12}
Let $S$ be a completely regular ordered semigroup. Then following
statements hold in $S$:
\begin{itemize}
 \item[(1)]For every $a \in S$ there is $h \in S$ such that
$a\leq aha, \;a \leq a^2h, \;\textrm{and} \;a \leq ha^2$.
 \item[(2)] For every $a \in S$ there is $h \in S$ such that
$a\hc ah \;\textrm{and} \;a \hc ha$.
 \item[(3)] For every $a \in S$ there is $e, f \in E_\leq(S)$ such that
$e \hc f$.
\end{itemize}
\end{Lemma}
\begin{proof}
$(1)$: Let $a \in S$.  Then there is $t \in S$ such that $a \leq
a^2ta^2$. Now $a  \leq a^3ta^2ta^2 \leq  a^3ta^2ta^2ta^3 \leq aha,
\;\textrm{where} \;h= a^2 ta^2 ta^2 ta^2$. Also $a \leq ha^2$ and $a
\leq a^2h$ are obvious.

$(2)$: Let $a \in S$. Then $a \leq aha$, where $h= a^2 ta^2 ta^2
ta^2$, follows from the proof of $(1)$. Now $ah= a( a^2 ta^2 ta^2
ta^2)=ua$ for some $u= a( a^2 ta^2 ta^2 ta) \in S$. Then from $(1)$
it follows that $a\hc ah $. Similarly $a\hc ha$.

$(3)$: This statement is  fairly straightforward.  $e=ah, \;f=ha \in
E_\leq(S)$ as in $(1)$ serves our purpose.
\end{proof}

For a semigroup   $S$  (without order),  the set $P(S)$ of all
finite subsets of $S$ is a semilattice ordered semigroup with the
operation $\cdot \;\textrm{and} \; \leq$ defined as follows:
$$\textrm{For} \;A, B \in P(F), \;A \cdot B = \{ab \mid a \in A, b
\in B\} \;\textrm{and} \;A \leq B \;\textrm{if and only if} \; A
\subseteq B \cite{bh1}.$$
\begin{Lemma}\label{0.1}
Let $S$ be a  regular semigroup. Then $P(S)$ is  regular.
\end{Lemma}

\begin{Lemma}\label{.1}
Let $S$ be a regular ordered semigroup. Then the following
statements  hold in $S$:
\begin{itemize}
 \item[(1)]
 For every  $a \in S, \; B(a )= (R(a) L(a)]$.
\item[(2)]
$(SA]\cap (AS] = (SA\cap AS]$ for any non empty subset $A$ of $S$.
\end{itemize}
\end{Lemma}
\begin{proof}
$(1)$ Let $x \in B(a)$. Since $S$ is ordered regular there is $s\in
S$ such that $x \leq asa$. Note that $a \in R(a)$ and this  yields
that $as \in R(a)$. Also $a \in L(a)$. Thus $asa \in R(a)L(a)$, so
$x \in R(a) L(a)$. Therefore $B(a )\subseteq (R(a) L(a)]$.

Again for some  $y \in (R(a) L(a)]$ there  is $s, t \in S$ such that
$y \leq asta$. Then $y \in (aSa]= B(a)$. Hence $B(a )= (R(a) L(a)]$.

$(2)$ First consider a nonempty subset $A$ of $S$.  Let $x \in
(SA]\cap (AS]$. Then there are $s,t \in S$ and $a,b \in A$ such that
$x \leq sa, \;x \leq bt$. Since $S$ is ordered regular $x\leq xzx$
for some $z \in S$ so that $x \leq bt z sa$. Now $b(tzsa)\in AS,
\;(btzs)a \in SA$. This yields that $btzsa \in SA\cap AS$. Therefore
$x \in (SA \cap AS]$ and hence $(SA]\cap (AS]\subseteq (SA\cap AS]$.
Also it is obvious that $(SA\cap AS] \subseteq (SA]\cap (AS]$. So
finally $(SA]\cap (AS] = (SA\cap AS]$.
\end{proof}

\begin{Theorem}\label{.2}
In a  regular ordered semigroup a nonempty subset $A$ of $S$ is a
bi-ideal of $S$ if and only if $A= (RL] $ for a right ideal $R$ and
a left ideal $L$ of $S$.
\end{Theorem}

Following  Xu and Ma \cite{XM} we define minimality of bi-ideals in
an ordered semigroup   as follows.
\begin{Definition}\cite{XM}
A bi-ideal $M$ of an ordered semigroup $S$ is called a minimal
bi-ideal if there is no non trivial bi-ideal $B$ such that $M
\subset  B$.
\end{Definition}
\begin{Theorem}\cite{XM}\label{t simple}
A bi-ideal $B$ of an ordered semigroup  $S$ is  minimal if and only
if $B$ is t-simple.
\end{Theorem}

\section{Characterizations of regular ordered semigroups  by their minimal bi-ideals}
In this section we characterize regular ordered by their minimal
bi-ideals.  We prove that a bi-ideal in a regular  ordered semigroup
is minimal if and only if it is a $\hc$-class. Also a regular
ordered semigroup is completely regular if and only if it is union
of minimal bi-ideals.

The following result makes  a natural analogy between  a bi-ideal in
a semigroup and a bi-ideal in an ordered semigroup.

\begin{Theorem}\label{lemma2}
Let $S$ be  an ordered semigroup $S$. Then for any bi-ideal $A$ of
$S$,  $P(A)$ is a bi-ideal  of $P(S)$.
\end{Theorem}
\begin{proof}
$(1):$ Let  $X \in P(A) P(S) P(A)$. Then there are $X_1, X_2 \in
P(A)$ and $Y \in P(S)$ such that $X=X_1 Y X_2$. Since $X_1, X_2 \in
P(A)$ we have that $X_1, X_2 \subseteq  A$ and so $X \subseteq ASA$.
Since $A$ is a bi-ideal of $S$ we have $X \in A$. Therefore $X\in
P(A)$. Hence $P(A)P(S)P(A) \subseteq P(A)$.

Next let $Y \in (P(A)]$. Then there is $Z \in P(A)$ such that $Y
\subseteq Z$. Then $Y \subseteq Z \subseteq A$, since $Z \in P(A)$.
Therefore $Y \in P(A)$ and so $(P(A)]= P(A)$. Hence $P(A)$ is a
bi-ideal of $P(S)$.
\end{proof}

\begin{Theorem}\label{th4}
Let $S$ be a regular ordered  semigroup. Then a non empty subset $B$
of $S$ is a minimal bi-ideal of $S$ if and only if $B=(RL]$ for some
minimal right-ideal $R$ and minimal left ideal $L$ of $S$.
\end{Theorem}
\begin{proof}
First suppose that $B$ is a minimal bi-ideal of $S$. Then  for every
$a \in B$, $B(a)= B$  and hence $B= (R(a)L(a)]$, by Lemma
\ref{.1}(1). Let $R$ be a right ideal of $S$ such that $R \subseteq
R(a)$. Since $S$ is regular, $(R(a) L(a)]= R(a) \cap L(a)$, by
Theorem \ref{1}. Now $(R L(a)]= R \cap L(a) \subseteq R(a) \cap
L(a)= B$. Also by Theorem \ref{.2},  $(R L(a)]$ itself a bi-ideal of
$S$ contained in $B$. By the minimality of $B$ it follows that $B=
(R L(a)]$. That is $R \cap L(a)= R(a) \cap L(a)$. Note that $a \in
R(a) \cap L(a)= R \cap L(a)$. Thus $a \in R$. Then for every $x \in
R(a), \;x \in R$. This implies that $R= R(a)$. Therefore $R(a)$ is a
minimal right ideal of $S$. Similarly $ L(a)$ is  a minimal left
ideal of $S$. Thus the condition is necessary.

Conversely, let $B$ be a nonempty  subset of $S$  such that $B=(RL]$
for a minimal left ideal $L$ and a minimal right ideal $R$ of $S$.
Then $B$ is a bi-ideal of $S$, by Theorem \ref{.2}. To prove the
minimality of $B$ let us choose a bi-ideal $B'$ of $S$ such that $B'
\subseteq B$. Then $(SB' ] \subseteq (SB]  \subseteq (S(RL]]
\subseteq (SL] \subseteq L$, since $L$ is a left ideal of $S$.

Likewise $ (B'S]\subseteq R$. Also  $(SB']$ and $(B'S]$ are left and
right ideals of $S$ respectively. Then the minimality of $L$ and $R$
yields that $(SB']= L$ and $(B'S]= R$. Therefore $B= (RL]=
((B'S](SB']] \subseteq (B'SB']= B'$,  by Theorem \ref{.3}. Thus
$B'=B$ and hence $B$ is a minimal bi-ideal of $S$.
\end{proof}
Then the following corollary follows from Theorem \ref{t simple} and
Theorem \ref{1}.

\begin{Corollary}\label{MB7}
Let $S$ be an ordered semigroup. If  $R$ is a minimal right ideal
and $L$ is a minimal left ideal of $S$ then $(RL]$ is a t-simple
ordered subsemigroup of $S$.
\end{Corollary}
By the Theorem \ref{1}, Theorem \ref{.2} and \ref{1}  we immediately
have the following corollary.
\begin{Corollary}\label{MB13}
Let $B$ be a bi-ideal of a regular ordered  semigroup $S$. Then $B$
is a minimal bi-ideal of $S$ if and only if $B$ is an intersection
of a minimal left ideal and a minimal right ideal.
\end{Corollary}

We have the following lemma that characterizes the minimality of
bi-ideal in respect of producing same principal bi-ideals.
\begin{Lemma}\label{MB3}
Let $S$ be an ordered semigroup. A bi-ideal $B$ of $S$ is minimal if
and only if $B(a)= B(b)$ for all $a, b \in B$.
\end{Lemma}
\begin{proof}
First assume that $B$ is a minimal bi-ideal of $S$. Let $a, b \in
B$. Then $B(a)= B= B(b)$.

Conversely, suppose that the given condition holds in $S$. Let $K$
be a bi-ideal of $S$ such that $K\subseteq B$. Let $z \in K$. Then
for every $x\in B,  \;B(x)=B(z)$ implies $x \in B(x)= B(z) \subseteq
K$ and so $K=B$. Hence $B$ is a minimal bi-ideal of $S$.
\end{proof}

Now we  introduce an equivalence relation which is determined in
respect of producing same principal bi-ideals.  Let $S$ be an
ordered semigroup. Define a relation $\beta$ on $S$ by:

$$\textrm{For} \;a, b \in S, \;a\beta b \Leftrightarrow \;B(a)= B(b).$$

It requires only routine verification to see that  $\beta$ is an
equivalence relation.
\begin{Lemma}\label{H beta}
The following conditions hold in an ordered semigroup $S$:
\begin{enumerate}
  \item \vspace{-.4cm}
 $\beta \subseteq \hc$.
 \item \vspace{-.4cm}
 If $S$ is regular then $\beta= \hc$.
 \end{enumerate}
\end{Lemma}
\begin{proof}
$(1)$  This is obvious.

$(2)$ First suppose that $S$ is regular. Let $a, b \in S$ be such
that $a \hc b$. Then $a \lc b$ and $a \rc b$ implies that
$L(a)=L(b)$ and $R(a)= R(b)$. Since $S$ is regular $B(a)= L(a)\cap
R(a)$, by Lemma \ref{.1} and Theorem \ref{1}. Thus $B(a) = L(b) \cap
R(b)= B(b) $, and so $a \beta b$. Hence by $(1)$ $\beta= \hc$.
\end{proof}

\begin{Theorem}\label{MB4}
Let $S$ be an ordered semigroup. Then every bi-ideal is a union of
$\beta$-class.
\end{Theorem}
\begin{proof}
Let $B$ be a bi-ideal of $S$ and $b \in B$. Let $a \in S$ be such
that $a \beta b$. Then $B(a)= B(b) \subseteq B$ implies that $a \in
B$. Thus the results follows.

\end{proof}

\begin{Theorem}\label{MB5}
Let $S$ be an ordered semigroup. A bi-ideal $B$ of $S$ is minimal if
and only if it is a $\beta$-class.
\end{Theorem}
\begin{proof}
First suppose that $B$ is a minimal bi-ideal of $S$. Let $a, b \in
B$. Then by the minimality of $B$ it follows from Lemma \ref{MB3}
that $B(a)=B(b)$, and this implies that $a\beta b$. Therefore $B$ is
a $\beta$-class.

Conversely, assume that a bi-ideal $B$ of $S$ is a $\beta$-class.
Choose a bi-ideal $K$ of $S$ such that $K \subseteq B$. Let $x\in B$
be arbitrary. Consider $y\in K$.  Then   $B(x)= B(y)$, since $x,y
\in B$. Therefore $x \in B(y) \subseteq K$ which implies $K=B$.
Hence $B $ is contained in a $\beta-$class and hence by Theorem
\ref{MB.1}, $B$ is a $\beta-$class.
\end{proof}

Immediately  we have the following corollary. It requires only
routine verification and so its proof is omitted.
\begin{Corollary}\label{MB15}
Let $B$ be a bi-ideal of a regular ordered  semigroup $S$. Then $B$
is minimal bi-ideal of $S$ if and only if $B$ is an $\hc$-class of
$S$.
\end{Corollary}
Let us consider the following example of \cite{ke98}.
\begin{Example}
Let $S = \{a, b, c, d, e \}$ be the ordered semigroup defined by the
multiplication and the order below:

\begin{center}\
\begin{tabular}{|l|l|l|l|l|l|l}
\hline
.& a &b &c &d &e\\
\hline
a &a &a &c &a &c \\
\hline
b &a &a &c &a &c \\
\hline
c &a &a &c &a &c \\
\hline
d &d &d &e &d &e\\
\hline
e &d &d &e &d &e\\
\hline
\end{tabular}
\end{center}
Define a relation $\leq$ on S as follows: $$\leq := \{(a, a), (a,
b), (a, c), (a, d), (a, e), (b, b), (b, c), (b, d), (b, e), (c, c),
(c, e), (d, d), (d, e), (e, e)\}.$$
\end{Example}
In this  example $(S, ., \leq)$ is an ordered semigroup. And in $S$,
$\{a,b \}$ and $\{a,b,c\}$ are bi-ideals of $S$. This shows that
$\{a,b,c \}$ is not a minimal bi-ideal of $S$. It is interesting to
note that $S$ is not regular and so not a complete regular. We now
see that every bi-ideal in a complete regular ordered semigroup is
bi-ideal.

\begin{Theorem}\label{MB31}
Let $S$ be a completely  regular  ordered semigroup. Then a bi-ideal
$B$ of $S$ is minimal if and only if $B(a)= B(e)$ for all $a \in B$
and $e\in E_\leq(B)$.
\end{Theorem}
\begin{proof}
First suppose that $B$ is minimal. Let $a\in B$ and $e \in
E_\leq(B)$.  Then $a,e \in B$. Since $S$ is regular we have $B(a)=
B(e)$, by Lemma \ref{MB3}.

Conversely, assume that the given conditions hold in $S$. Let $K$ be
a bi-ideal of $S$ such that $K\subseteq B$. Let $z \in K$ and $b \in
B$. Since $S$ is completely regular there is $h \in S$ be such that
$b \hc bh$, by Lemma \ref{cr12}(2). So $B=B(b)=B(bh)$. Also from the
proof of  Lemma \ref{cr12}(3) we have $bh \in E_\leq(S)$, infact $bh
\in E_\leq(B)$. So by given condition $B(z)= B(bh)$. Therefore $z
\in B(z)= B$, and so $K=B$. Hence $B$  is  minimal  bi-ideal  of
$S$.
\end{proof}

\begin{Corollary}\label{MB32}
Let $S$ be a completely  regular  ordered semigroup. Then a bi-ideal
$B$ of $S$ is minimal if and only if $B(e)= B(f)$ for all $e, f\in
E_\leq(B)$.
\end{Corollary}

\begin{proof}
This is obvious.
\end{proof}
Next we discuss about the bi-simplicity of an ordered semigroup.
\begin{Definition}
An ordered semigroup $S$ is called bi-simple if $S$ has no proper
bi-ideal.
\end{Definition}

In the following theorem bi-simplicity of regular ordered semigroup
has been described by its any  two ordered idempotents.
\begin{Theorem}
Let $S$ be a regular ordered semigroup. Then $S$ is bi-simple if and
only if for every $e, f \in E_\leq(S), B(e)= B(f)$.
\end{Theorem}
\begin{proof}
Suppose that  $S$ is bi-simple. Consider $e, f \in E_\leq(S)$.
Clearly $B(e)= S= B(f)$.

Conversely, assume that the given condition hold in $S$ and  choose
a bi-ideal $B $ of $S$. Let  $a \in S$ and $b \in B$. Since $S$ is
regular there are $x,y \in S $ such that $a \leq axa$ and $b \leq
byb$. Clearly $ax,xa,by \;\textrm{and} \;yb\in E_\leq(S)$. Now by
given condition we have $B(ax)=B(by) \;\textrm{and} \;B(xa)=B(yb)$.
This yields that $ax \hc by  \;\textrm{and} \;xa \hc yb$, in
otherwords $ax \rc by  \;\textrm{and} \;xa \lc yb$.  Now $ax \rc by
$ gives $ax \leq byz \;\textrm{and} \;by \leq axw$ for some $z,w \in
S$. So from $a \leq axa$ and $b\leq byb$ we have $a\leq byza$ and $b
\leq axwb$, which gives that $a \rc b$. In like manner $a \lc b$
follows from $yb \lc xa$. Thus $a \hc b$ and so $B(a)= B(b)$, by
Lemma \ref{H beta}. So $a \in B(b)=B$. Hence $S=B$. This completes
the proof.
\end{proof}
Every minimal bi-ideal is a bi-simple ordered semigroup. It is
interesting to note down that there are bi-ideals which are neither
left ideal nor a right ideal, but an ordered semigroup $S$ is
bi-simple if and only if it is both left and right simple. Thus we
have the following theorem.
\begin{Theorem}\label{bi simple}
The following conditions are equivalent  on an ordered semigroup
$S$:
\begin{enumerate}
  \item \vspace{-.4cm}
 $S$ is bi-simple;
 \item \vspace{-.4cm}
  $S$ is t-simple ordered semigroup;
  \item \vspace{-.4cm}
For every $a \in S, \;S=B(a) $;
\item \vspace{-.4cm}
For every $a \in S, \;S=L(a) \;\textrm{and} \;S=R(a)$.
 \end{enumerate}
\end{Theorem}
\begin{proof}
$(1)\Rightarrow (2)$: This is obvious.

$(2)\Rightarrow (3)$: This implication follows from Theorem \ref{0}.

$(3)\Rightarrow (4)$: This follows from from the fact that every
left ideal and every right ideal are bi-ideals.

$(4)\Rightarrow (1)$: Let the given conditions hold in $S$. Consider
a bi-ideal $B$ of $S$. Let $a \in B$. Then $B(a)= B$. Now $S=L(a^2)$
and $S= R(a^2)$, by condition (4). Let $x \in S$. Then $x \in
L(a)\cap R(a)$. This implies that $x \leq sa^2$  for some $s\in
S^1$. Since $sa \in S$ there is $t \in S^1$ such that $sa \leq
a^2t$. Thus $x \leq sa^2$ implies that $x \leq a^2ta=a(at)a$. Since
$a(at)a \in BSB$ and $B$ is a bi-ideal of $S$ we have that $a^2ta
\in B$. Thus $x \in B$, and hence $S=B$. This shows that $S$ is
bi-simple.
\end{proof}

\begin{Theorem}\label{mini bi}
Let $ S$ be an ordered semigroup. Then a bi-ideal  $B$ is minimal if
and only if it is bi-simple.
\end{Theorem}

\begin{proof}
First suppose that  $B$ is  a minimal bi-ideal of $S$. Consider a
bi-ideal $T$ of $B$. Let $x \in  T$. Then $(xBx] \subseteq  (TBT]
\subseteq T$. This  implies that $(xBx] \subseteq T \subseteq B$.
Also  $(xBx]$ is a bi-ideal of $S$, by Lemma . Then the minimality
of $B$ yields   that $(xBx] = B$,  and so $T = B$. Therefore  $B$ is
bi-simple.

Conversely assume that $B$ is bi-simple. Consider a bi-ideal  $Y$ of
$S$ such that $Y \subseteq B$. Choose $y \in Y$ arbitrarily. Then by
Lemma $(yBy]$ is a bi-ideal of $S$. Since $B$ is bi-simple,  $(yBy]
= B$. Then $B \subseteq (YBY] \subseteq (YSY] \subseteq Y$.
Therefore $Y = B$ and hence  $B$ is a minimal bi-ideal of $S$.
\end{proof}

In the next theorem we  characterize  completely regular ordered
semigroups in terms of  their bi-ideals.
\begin{Theorem}
An ordered semigroup $S$ is a completely regular ordered semigroup
if and only if the following conditions hold in $S$:

\begin{itemize}
\item[(1)]For every bi-ideal $B \;\textrm{of} \;S$ there is some $e \in
E_\leq(S)$ such that $ \;B = B(e)$.
\item[(2)]For every $x \in B, \;B(x^2)=B(e)$.
\end{itemize}
\end{Theorem}
\begin{proof}
Let $S$ be a  completely regular ordered semigroup. Consider a
bi-ideal $B$ of $S$.   Choose $a \in B$. Then by the Theorem
\ref{cr12}, there is $h \in S$ such that $a \leq aha, \;a \leq a^2h$
and $a \leq ha^2$. Then $B(a)= (aSa] \subseteq (aha S a^2h]\subseteq
(ah S ah]= B(e)$, where $e =ah \in E_\leq(S)$. Also $B(e)= (eSe]=(ah
S ah]$. Now $h=a^2ta^2ta^2ta^2$, by the proof of Lemma
\ref{cr12}(1), and so $B(e)=(ah S ah] \subseteq (aSa]=B(a) $.
Therefore $B(a)= B(e)$. Thus condition $(1)$ follows.

For condition $(2)$ let $x \in B$. Clearly  $B(x^2)=B=B(a)$. So by
condition $(1)$ we have $B(x^2)= B(e)$.

Conversely, assume  that the given conditions  hold in $S$. Let  $a
\in S$. Consider the bi-ideal $B(a)$ of $S$.  Then  there is $e \in
E_\leq (S)$ be such that $B(a) = B(e)=B(a^2)$. Then $B(a)= B(a^2)=
(a^2 \cup a^4 \cup a^2 S a^2]$. Thus $a \leq a^2$ or $a \leq a^4$ or
$a \in (a^2 Sa^2]$. Hence in either case $S$ is completely regular
ordered semigroup.
\end{proof}
\begin{Corollary}
An ordered semigroup  $S$ is  completely  regular  if and only if
for every $a \in S, \;B(a)= B(a^2)$.
\end{Corollary}

\begin{Corollary}
Let  $S$ be an   completely  regular  ordered semigroup. Then every
bi-ideal of $S$ is principal bi-ideal generated by some ordered
idempotent.
\end{Corollary}
 Then we have the  following corollary which follows from Theorem \ref{MB31}.
\begin{Corollary}\label{cr minimal}
Let  $S$ be an   completely  regular  ordered semigroup. Then every
bi-ideal of $S$ is minimal.
\end{Corollary}

In the following theorem  completely regular ordered semigroups are
characterized by the minimality of their   bi-ideals.
\begin{Theorem}
Let  $S$ be a regular ordered semigroup. Then $S$  is  completely
regular if and only if $S$  is union of its bi-ideals.
\end{Theorem}
\begin{proof}
First suppose that $S$ is completely regular. Since $S$ is regular
$\hc= \beta$, by Lemma \ref{H beta}. Then by Theorem \ref{2}, $S$ is
union of $\beta-$classes and so $S$ is union of minimal bi-ideals,
by Theorem \ref{MB5}. Therefore from Corollary \ref{cr minimal}, $S$
is union of bi-ideals of $S$.

Conversely, assume that $S$ is union of its minimal bi-ideals. Then
$S$ is union of its t-simple ordered subsemigroups, by Theorem
\ref{bi simple} and Theorem \ref{mini bi}. Hence by Theorem \ref{2}
$S$ is completely regular.

\end{proof}


\begin{thebibliography}{9}

\bibitem{bh1}
 A.K. Bhuniya and  K. Hansda, \emph{Complete semilattice of ordered
semigroups}, Communicated.





\bibitem{Cao00}
Y. Cao and  X. Xu, {On minimal and maximal Left Ideals in ordered
semigroups}, \emph{Semigroup Forum}, \textbf{60(200)}, 202-207(2000)

\bibitem{Cao02}
Y. Cao,\emph{Characterizations of Regular Ordered Semigroup  by
Quasi-ideals}, {Vietnam Journal of Mathematics}, \textbf{30(3)},
239-250(2002)


\bibitem{ra}
R. A. Good and D. R. Hughes, {Associated groups for a semigroup},
\emph{Bull. Amer. Math. Soc.} \textbf{58}(1952), 624-625.


\bibitem{bh2}
K. Hansda, \emph{Bi-ideals in Clifford ordered semigroup}, accepted
for publication in Discussiones Mathematicae.



\bibitem{ke91}
N. Kehayopulu,  {Note on Green's relation in ordered semigroup} ,
\emph{Math. Japonica} \textbf{36}(1991), 211-214.


\bibitem{ke92}
N. Kehayopulu,  \emph{On completely regular poe-semigroups}, {Math.
Japonica} \textbf{37}(1992), 123-130.

\bibitem{ke921}
N. Kehayopulu,   {On  regular duo ordered semigroups}, \emph{Math.
Japonica} \textbf{37}(1992), 535-540.


\bibitem{ke98}
N. Kehayopulu,   {On  completely regular ordered semigroups},
\emph{Scinetiae Mathematicae} \textbf{1(1)}(1998), 27-32.

\bibitem{ke02}
N. Kehayopulu, J. S. Ponizovskii and M. Tsingelis, \emph{Bi-ideals
in ordered semigroups and ordered groups},{Journal of Mathematical
Sciences} \textbf{112, (4)}(2002). ISSN: 1573-8795




\bibitem{Ke2006}
N.  Kehayopulu, {Ideals and Green's relations in ordered
semigroups}, \emph{International Journal of Mathematics and
Mathematical Sciences }, \textbf{Article ID 61286}, 1-8(2006).







\bibitem{laj60}
S. Lajos, {On (m, n)-ideals of semigroups}, \emph{Abstract of Second
Hunger. Math. Congress I }, \textbf{}(1960), 42-44.

\bibitem{laj69}
S. Lajos, \emph{On the bi-ideals in Semigroups}, {Proc. Japan
Acad.}, \textbf{45(8)}(1969),710-712.  doi:10.3792/pja/1195520625

\bibitem{laj70}
S. Lajos, {Notes on semilattices of groups}, \emph{Proc. Japan
Acad.}, \textbf{46(2)}(1970), 151- 152. doi:10.3792/pja/1195520460




\bibitem{lt}
S. Lee and  S. Kang, \emph{Characterization of regular po-
semigroup}, {Commu. Korean Math. Soc}, \textbf{14(4)}(1999),
687-691.


\bibitem{Saito2}
T. Saito,  {Regular elements in an ordered semigroups},
\emph{Pacific J. Math.}, \textbf{13} (1963), 263-295.

\bibitem{Saito3}
T. Saito,  {Ordered completely regular semigroups.}, \emph{Pacific
J. Math.}, \textbf{14} (1964), 295-308.


\bibitem{stein}
O. Steinfeld, \emph{Quasi-ideals in rings and regular semigroups},
\emph{Akademiai Ki- ado, Budapest } \textbf{}(1958). doi:
10.1007/BF02572529


\bibitem{XM}
X. Xu and J. Ma, {A note on Minimal Bi-ideal in ordered semigroups},
\emph{Southeast Asian Bulletin of Mathematics}, \textbf{27},
149-154(2003).



\end{thebibliography}
\end{document}